\documentclass[a4paper]{amsart}

\usepackage{mathtools}
\usepackage{amsmath}
\usepackage{amsthm}
\usepackage{amssymb}
\usepackage{dsfont}
\usepackage[english]{babel}
\usepackage[utf8x]{inputenc}
\usepackage{enumerate}
\usepackage{pdflscape}
\usepackage{float}
\usepackage{hyperref}
\usepackage{todonotes}
\usepackage{overpic}
\usepackage{tikz}
\usetikzlibrary{decorations.markings,decorations.pathreplacing}
\newtheoremstyle{pl}
{3pt}
{3pt}
{\itshape}
{}
{\scshape}
{.}
{.5em}
{}

\newtheoremstyle{pl*}
{3pt}
{3pt}
{\itshape}
{}
{\bfseries}
{.}
{.5em}
{}

\newtheoremstyle{pl**}
{3pt}
{3pt}
{\itshape}
{}
{\bfseries}
{.}
{.5em}
{}

\newtheoremstyle{mythm}
{3pt}
{3pt}
{\itshape}
{}
{\bfseries}
{.}
{.5em}
{\thmnote{#1 }#3}

\newtheoremstyle{myappendix}
{3pt}
{3pt}
{\itshape}
{}
{\scshape}
{.}
{.5em}
{\thmnote{#1 }#3}

\newtheoremstyle{df}
{3pt}
{3pt}
{\normalfont}
{}
{\scshape}
{.}
{.5em}
{}

\newtheoremstyle{rm}
{3pt}
{3pt}
{\normalfont}
{}
{\scshape}
{.}
{.5em}
{}

\theoremstyle{pl}
\newtheorem{thm}{Theorem}[section]
\newtheorem{lem}[thm]{Lemma}			

\newtheorem{pro}[thm]{Proposition}

\newtheorem*{qn*}{Question}
\newtheorem*{lem*}{Lemma}

\theoremstyle{pl*}
\newtheorem{thm*}{Theorem}
\newtheorem{pro*}[thm*]{Proposition}
\newtheorem{cor*}[thm*]{Corollary}

\theoremstyle{pl**}
\newtheorem*{thm**}{Main Theorem}

\theoremstyle{mythm}

\theoremstyle{myappendix}

\theoremstyle{df}
\newtheorem*{dfn}{Definition}
\newtheorem{ex}[thm]{Example}

\theoremstyle{rm}
\newtheorem{rmk}[thm]{Remark}

\newcommand{\mc}[1]{
\mathcal{#1}
}
\newcommand{\mb}[1]{
\mathbb{#1}
}

\begin{document}

\title{Hyperbolic Heegaard splittings and Dehn twists}

\author[Feller]{Peter Feller}
\address{Universit\'e de Neuch\^atel, Switzerland}
\email{peter.feller@math.unine.ch}

\author[Sisto]{Alessandro Sisto}
	\address{Department of Mathematics, Heriot-Watt University, Edinburgh, UK}
	\email{a.sisto@hw.ac.uk}

\author[Viaggi]{Gabriele Viaggi}
\address{Department of Mathematics, University of Pisa, Italy}
\email{gabriele.viaggi@unipi.it}


\begin{abstract}
We consider the family of Heegaard splittings of genus $g$ at least two which are defined via a gluing map that is the $n$-th power of the Dehn twist along a curve that satisfies a natural topological assumption, namely pared acylindricity. We show that if $n$ is at least 14, then the Heegaard splitting has a hyperbolic metric for which the simple closed curve defining the Dehn twist is a closed geodesic of length at least $1.24/(n^2g)$ and at most $37.5/n^2$.
\end{abstract}

\maketitle

\section{Introduction}

A Heegaard splitting is a closed orientable 3-manifold 
\[
M_f=H\cup_{f:\partial H\to\partial H}H
\] 
obtained by gluing two copies of a
handlebody $H$ along an orientation preserving diffeomorphism of the boundary $f:\partial H\to\partial H$. A classical result due to Heegaard states that every closed connected orientable 3-manifold is diffeomorphic to some $M_f$. A general, challenging problem is to convert the topological and dynamical information about the gluing map $f$ into topological and geometric properties of $M_f$ (most relevantly for this paper; see e.g.~\cite{Hempel:Heegaard,twist-distance} for the topological side, and \cite{Na05,NS09,FSVa,HV,HJ22} for the geometric side).

In this article, we consider gluing maps $f$ that are some power of a Dehn twist along a simple closed curve $\gamma\subset\partial H$, and specifically we prove the following.

\begin{thm**}
Let $H$ be a handlebody of genus $g\geq 2$. Let $\gamma\subset\partial H$ be an essential simple closed curve such that $(H,\gamma)$ is {\rm pared acylindrical}. Let $\tau_\gamma$ be a (left or right) Dehn twist along $\gamma$.
If $n\ge 14$, then $M_{\tau_\gamma^n}$ admits a hyperbolic metric such that the curve $\gamma\subset M_{\tau_\gamma^n}$ is a geodesic and its length $\ell_{M_{\tau_\gamma^n}}(\gamma)$ satisfies
\[
\frac{1.24}{gn^2}<\ell_{M_{\tau_\gamma^n}}(\gamma)
<\frac{37.5}{n^2}.
\]
\end{thm**}
Pared acylindricity is a topological condition introduced by Thurston in the study of deformation spaces of hyperbolic metrics
. We provide the exact definition below, and at this point only note that it is satisfied when, in the curve graph of the boundary surface of the handlebody, the curve $\gamma$ is at distance at least 3 from the disk set of the handlebody; see~\cite{Hempel:Heegaard} for definitions of these terms. In fact, in the distance at least 3 case, the manifold $M_{\tau_\gamma^n}$ is hyperbolic for $n\geq 3$, by \cite[Lemma 5.3]{twist-distance} and in view of Thurston's hyperbolization and \cite{Hempel:Heegaard}, but no bounds were previously known on the length of $\gamma$.

The theorem is curated to please the eye. For a more precise, albeit more technical version, see Theorem~\ref{thm} below. 

In~\cite{FSVb}, we establish hyperbolicity and length bounds for a much larger class of Heegaard splittings using subsurface projections. The reasons why we deal with the case of powers of Dehn twists in this separate paper rather than as part of \cite{FSVb} are the more elementary nature of the arguments for the powers of Dehn twist case, which allows a short exposition of the key idea, and, more importantly, the fact that the constants in the present case are explicit. In the more general setup, the strategy cannot yield explicit constants; see also the last bullet point below.

We list some features of Theorem~\ref{thm}.
\begin{itemize}
\item{The lower bound on the power on the Dehn twist and the upper bound on the length of the geodesic are {independent} of the genus of the handlebody.}
\item{There are lots of simple closed curves $\gamma\subset\partial H$ that satisfy the assumption of Theorem~\ref{thm}. To make this precise, consider the space $\mc{PML}$ of projective measured laminations on $\partial H$. The space $\mc{PML}$ is homeomorphic to a sphere of dimension $6g-7$, where $g$ is the genus of $H$, and it is a certain completion of the space of simple closed curves on $\partial H$. There is an open dense subset $\mc{D}\subset\mc{PML}$ of full Lebesgue measure, the so-called Masur domain (see Lecuire~\cite{Le05}), such that $(H,\gamma)$ is pared acylindrical for every simple closed curve $\gamma\in\mc{D}$.
}
\item
In a separate article \cite{FSVb}, we combine the ideas of this article and of~\cite{FSVa} with 
 tools introduced by Masur and Minsky~\cite{MasurMinsky:I},~\cite{MasurMinsky:II}, Minsky~\cite{M10}, and Brock, Canary, and Minsky~\cite{BCM} around the solution of the Ending Lamination Conjecture to establish hyperbolicity and length bounds (as given by the Length Bound Theorem in \cite{BCM}) for a vast class of Heegaard splittings with large subsurface projections. The ideas presented here correspond to the most basic case of large annular projections.
\end{itemize}

\subsection*{Main ingredients of the proofs}
In order to prove our main theorem, we will regard $M_{\tau_\gamma^n}$ as obtained by Dehn filling from the double of $H-\gamma$, which is a hyperbolic manifold that we denote by $\mathbb M_\gamma$. By work of Hodgson-Kerckhoff \cite{HK:universal} and Futer-Purcell-Schleimer \cite{FPS22}, one can understand the geometry of $M_{\tau_\gamma^n}$ provided that one understands the geometry of the cusp torus; the precise version of this is stated below.
In our case, the manifold $\mathbb M_\gamma$ has a reflective isometry, which is the special feature of $M_{\tau_\gamma^n}$ compared to manifolds obtained using more complicated gluing maps. This symmetry forces two natural curves $\alpha$ and $\beta$ on the cusp torus, whose classes form a basis of its first homology, to admit orthogonal geodesic representatives. As a consequence, the normalized length of the filling slope can be computed exactly in terms of the lengths $a$ and $b$ of these two geodesics; see Subsection~\ref{subsec:normalized_length}. It thus remains to bound $a$ and $b$: a lower bound for both is provided by Adams' waist size theorem \cite{Adams_19_WaistSize} applied to the maximal cusp, while the upper bounds are the content of Proposition \ref{prop}. The upper bound for $a$ uses the Gauss-Bonnet theorem together with, to improve the estimate, a packing estimate for cusp neighborhoods on hyperbolic surfaces from \cite{Boroczky}. The upper bound for $b$ relies on the reflective isometry once more, together with a hyperbolic geometry lemma on a sequence of horospheres and geodesics, Lemma \ref{lem:2.5}.

\subsection*{Acknowledgements}
PF acknowledges financial support by the SNSF
Grant 181199. GV acknowledges financial support of the DFG grants 427903332 and 390900948.
We are most grateful for the very careful and detailed feedback of an anonymous referee, which improved the paper and the main result considerably. Concretely, they found a factor 2 error concerning injectivity radii that led to errors in the original argument, and gave multiple suggestions for improvements of our constants, including correctly calculating the exact constants in Theorem~\ref{thm}. 

\noindent\emph{AI disclosure.} We also acknowledge the use of Claude AI. For the second version of this paper, we uploaded the paper to Claude AI~\cite{claude2026} for a correctness check and independent verification of the numerical calculations. In the process it spotted an error in Example~\ref{ex:optimal} and proposed the solution to said error.

\section{A precise version of the theorem}

We recall the definition of pared acylindricity and state a refined, albeit more technical version of the Main Theorem.

\begin{dfn}[Pared Acylindrical~\cite{Thu86,Thu86(3)}]\label{def:pa}
Let $H$ be a 
handlebody of genus $g\ge 2$. Let $\gamma\subset\partial H$ be an essential simple closed curve. Let $A$ be a tubular neighborhood of $\gamma$ in $\partial H$. We say that $(H,\gamma)$ is {\em pared acylindrical} if
\begin{itemize}
\item{the inclusion $\partial H-\gamma\subset H$ is $\pi_1$-injective,}
\item{every essential map $(S^1\times[0,1],S^1\times\{0,1\})\to (H,A)$ is homotopic as a map of pairs into $A$, and}
\item{every essential map $(S^1\times[0,1],S^1\times\{0,1\})\to (H,\partial H-A)$ is homotopic as a map of pairs into $\partial H$.}
\end{itemize}
\end{dfn}

\begin{thm}\label{thm}
Let $H$ be a 
handlebody of genus $g\geq 2$. Let $\gamma\subset\partial H$ be an essential simple closed curve such that $(H,\gamma)$ is {\rm pared acylindrical}. Let $\tau_\gamma$ be a (left or right) Dehn twist along $\gamma$.

If $n\ge 14$, then $M_{\tau_\gamma^n}$ admits a hyperbolic metric for which the curve $\gamma\subset M_{\tau_\gamma^n}$ is a geodesic of length
\begin{equation}\label{eq:lengthestimate}
\frac{2\pi}{n^2\left(\frac{3(g-1)}{r}+\frac{r}{3}\right)+16.17}<\ell_{M_{\tau_\gamma^n}}(\gamma)
<\frac{2\pi}{n^2\frac{r}{2}+\frac{2}{r}-28.78},
\end{equation}
where $r=\frac{\sqrt[4]{2}}{2}=2^{-3/4}$.
\end{thm}

Theorem~\ref{thm} recovers the Main Theorem.

\begin{proof}[Proof of the Main Theorem]
For the lower bound, we must check that, for all $g\geq 2$ and $n\geq 14$,
\[\frac{1.24}{gn^2}<\frac{2\pi}{n^2\left(\frac{3(g-1)}{r}+\frac{r}{3}\right)+16.17},\]
which is equivalent to
$1.24\left(\frac{3(g-1)}{r}+\frac{r}{3}+\frac{16.17}{n^2}\right)<2\pi g$.
Using $n\geq 14$, we have $\frac{r}{3}+\frac{16.17}{n^2}<0.199+0.083=0.282$, so it suffices to note
\[1.24\cdot\frac{3(g-1)}{r}+1.24\cdot 0.282<6.257(g-1)+0.35<2\pi g,\]
where the last inequality holds since $6.257<2\pi$.

For the upper bound, note that
$\frac{2\pi}{n^2\frac{r}{2}+\frac{2}{r}-28.78}<\frac{37.5}{n^2}$
is equivalent to
\[n^2\left(37.5\cdot\frac{r}{2}-2\pi\right)>37.5\left(28.78-\frac{2}{r}\right),\]
which, since $37.5\cdot\frac{r}{2}-2\pi>4.865$ and $37.5\left(28.78-\frac{2}{r}\right)<953.12$, holds whenever $n^2\geq 196>\frac{953.12}{4.865}$, i.e.~whenever $n\geq 14$.
Thus, the Main Theorem is implied by Theorem~\ref{thm}.
\end{proof}
In the rest of this section, we prove Theorem~\ref{thm}.
\subsection{Hyperbolization of the drilled double}
\label{subsec:hyperbolization}
Observe that $M_{\tau_{\gamma}^n}-\gamma$ is (homeomorphic to) the double of $H-\gamma$.
By Thurston's Hyperbolization for Haken manifolds (\cite[Theorem~1.42]{Kapovich:hyperbolization}), as $(H,\gamma)$ is pared acylindrical, the manifold $M_{\tau_{\gamma}^n}-\gamma$ admits a complete finite volume hyperbolic metric which is itself the double of a complete hyperbolic metric with finite volume and totally geodesic boundary on $H-\gamma$ (see also~\cite[Theorem~2]{BonahonOtal_04}). We fix such a hyperbolic structure on $M_{\tau_{\gamma}^n}-\gamma$ and denote it by $\mb{M}_\gamma$. This hyperbolic structure is unique up to isometry by Mostow--Prasad rigidity (and, likewise, the hyperbolic structure with totally geodesic boundary on $H-\gamma$ is unique up to isometry, as can be seen by doubling and applying Mostow--Prasad rigidity).

Since the metric on $\mb{M}_\gamma$ is the double of a metric on $H-\gamma$, the \emph{doubling involution} $\sigma\colon\mb{M}_\gamma\to\mb{M}_\gamma$, which exchanges the two copies of $H-\gamma$ and fixes their common boundary pointwise, is an isometry. This reflective symmetry will be used twice below, in Subsection~\ref{subsec:normalized_length} and in the proof of Proposition~\ref{prop}.

\subsection{The cusp and its waist}
\label{subsec:the_cusp}
The hyperbolic structure $\mb{M}_\gamma$ has one (rank $2$) \emph{cusp}, i.e.~there exists an open subset $\mb{T}(\gamma)$ of $M_{\tau_{\gamma}^n}-\gamma$ that is isometric to a quotient of a horoball $\mc{O}\subset \mb{H}^3$ by a discrete torsion free group of isometries isomorphic to $\mb{Z}^2$ that stabilize $\mc{O}$.

By ``expanding the cusp until it first touches itself on its boundary'', i.e. by choosing $\mb{T}(\gamma)$ maximal with respect to inclusion, we know that the shortest non-null homotopic curve in the flat boundary torus $\partial \mb{T}(\gamma)$ has length at least $\sqrt[4]{2}$ by work of Adams~\cite{Adams_19_WaistSize}. Here we use that Adams' theorem applies to $\mb{M}_\gamma$: the theorem allows for exactly three exceptional manifolds with a cusp of waist size at most $\sqrt[4]{2}$, namely the figure-eight knot complement, the $5_2$ knot complement, and the manifold obtained by $(2,1)$-surgery on one component of the Whitehead link. None of them is isometric to $\mb{M}_\gamma$. Indeed, each of the three exceptional manifolds has first Betti number $1$, whereas $\mb{M}_\gamma$ is the complement of a knot in the double of $H$, which is diffeomorphic to $\#_g(S^2\times S^1)$; hence $b_1(\mb{M}_\gamma)\geq g\geq 2$.

We set
\[r\coloneqq\frac{\sqrt[4]{2}}{2}=2^{-3/4},\]
so that every essential closed curve on $\partial\mb{T}(\gamma)$ has length at least $2r$.

As $\mb{M}_\gamma$ is the double of $H-\gamma$, the cusp $\mb{T}(\gamma)$ is itself the double of $U=\mb{T}(\gamma)\cap (H-\gamma)$. The boundary of $U$ decomposes as a Heegaard surface part $A=U\cap(\partial H-\gamma)$, which is the union $A=A_1\cup A_2$ of two cuspidal neighborhoods of the ends of the finite area hyperbolic surface $\partial H-\gamma$, and a handlebody part $\partial U-A$ which is a properly embedded annulus in $H-\gamma$. 

The boundary of $\mb{T}(\gamma)$ is the double of the annulus $\partial U-A$. We choose and fix for the rest of the paper two simple closed curves in  $\partial\mb{T}(\gamma)$ whose homology classes form a basis for the first integer homology of $\partial\mb{T}(\gamma)$:
\begin{itemize}
\item{The curve $\alpha$ is a component of $\partial A$.}
\item{The curve $\beta$ is the double of an arc $\kappa$ joining the two boundary components of the annulus $\partial U-A$. Note that the homology class of $\beta$ does not depend on the choice of the arc.}
\end{itemize}
See Figure~\ref{fig:cusp_torus} for an illustration.

\begin{figure}[h]
    \begin{overpic}[scale=2]{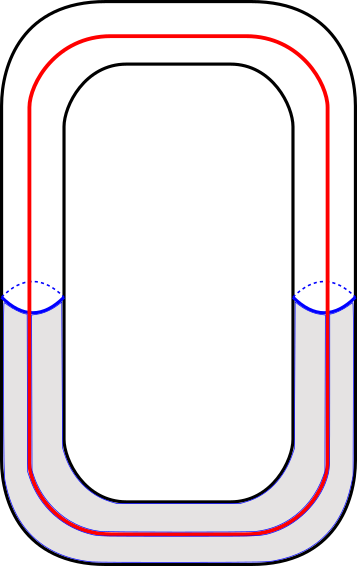}
    \put(-20,20) {$\partial U-A$}
    \put(-25,80) {$\sigma(\partial U-A)$}
    \put(-5,45) {{\color{blue} $\alpha$}}
    \put(65,50) {{\color{red} $\beta=\kappa\cup\sigma(\kappa)$}}
    \end{overpic}
    \caption{The flat torus $\partial\mb{T}(\gamma)$ is symmetric with respect to the isometric involution $\sigma$. As explained in Subsection \ref{subsec:normalized_length}, the geodesic representatives of $\alpha$ and $\beta=\kappa\cup\sigma(\kappa)$ are orthogonal, of lengths $a$ and $b$ respectively.}\label{fig:cusp_torus}
\end{figure}

\subsection{Dehn Filling}
$\mb{T}(\gamma)\cup\gamma$ is an open solid torus and a tubular neighborhood of $\gamma$ in $M_{\tau_\gamma^n}$ and, hence, 
$M_{\tau_\gamma^n}$ is obtained from $\mb{M}_\gamma-\mb{T}(\gamma)$ by Dehn filling along a suitable slope. This means that we recover $M_{\tau_\gamma^n}$, up to diffeomorphism, as the result of gluing together  $\mb{M}_\gamma-\mb{T}(\gamma)$ with a standard solid torus $D^2\times S^1$ along a diffeomorphism $\phi_n:\partial(D^2\times S^1)\to\partial \mb{T}(\gamma)$. The diffeomorphism type of the resulting manifold is completely determined by the homotopy class of the curve $\mu_n=\phi_n(\partial D^2\times\{\star\})$ on the torus $\partial\mb{T}(\gamma)$. This homotopy class is called the \emph{slope} of the Dehn filling. In our case, the slope is the (unique up to orientation reversal) non-trivial homotopy class
$\mu_n\subset\partial\mb{T}(\gamma)$
given by curves that are null-homotopic in the solid torus $\mb{T}(\gamma)\cup\gamma$.

Homotopy classes of simple closed curves on a torus $\partial\mb{T}(\gamma)$ are parametrized by the primitive homology classes in $H_1(\partial\mb{T}(\gamma),\mb{Z})$. From here on, for curves on a torus, we do not distinguish between them, their homotopy classes and their homology classes in our notation.
 The slope defined by $\mu_n$ along which Dehn filling of $\mb{M}_\gamma$ yields $M_{\tau_\gamma^n}$ has, up to correct choice of orientations which we suppress, a simple expression in terms of the homology basis $\alpha,\beta$, namely
$\mu_n=n\alpha+\beta$.
To see this, note that if $n=0$, indeed $\mu_0=\beta$ since $\beta$ bounds a disc in the solid torus $\mb{T}(\gamma)\cup\gamma$.
For general $n$, $\mu_n$ is obtained from $\mu_0$ by applying $n$ Dehn twists along $\alpha$ as e.g.~argued in~\cite[Proof of Theorem 2]{lickorish62}; hence, $\mu_n=n\alpha+\mu_0=n\alpha+\beta$ as desired.

In order to prove our theorem, we apply the universal Dehn filling theorem of Hodgson and Kerckhoff \cite{HK:universal} to deform the hyperbolic structure of $\mb{M}_\gamma$ to a hyperbolic structure on $M_{\tau_\gamma^n}$. We use the following statement of the filling theorem obtained by combining the work of Hodgson and Kerckhoff and results of Futer, Purcell, and Schleimer. Recall that the {\em normalized length} of a curve $\mu\subset\partial\mb{T}(\gamma)$, where $\mb{T}(\gamma)$ is a cusp of some finite volume hyperbolic $3$-manifold, is defined by
\[
L(\mu):=\frac{{\rm Length}(\bar{\mu})}{\sqrt{{\rm Area}(\partial\mb{T}(\gamma))}}
\]
where $\bar\mu$ is a (flat) geodesic representative of $\mu$ on $\partial\mb{T}(\gamma)$ (the intrinsic metric of $\partial\mb{T}(\gamma)$ is locally isometric to a horosphere $\mc{H}$ in $\mb{H}^3$ which is isometric to $\mb{R}^2$).
\begin{thm}[{\cite[Theorem 1.1]{HK:universal} and \cite[Corollary 6.13]{FPS22}}]\label{thm:filling}
Let $\Gamma\subset M$ be a knot in a 3-manifold $M$. Suppose that $M-\Gamma$ has a complete finite volume hyperbolic metric for which the {\rm normalized length} $L$ of the meridian of $\Gamma$ is at least $7.823$. Then $M$ has a complete hyperbolic metric for which $\Gamma$ is a geodesic of length 
\[
\frac{2\pi}{L^2+16.17}<\ell_M(\Gamma)<\frac{2\pi}{L^2-28.78}.
\]
\end{thm}
To establish Theorem~\ref{thm}, we provide bounds on the normalized length of $\mu_n
$.

\subsection{Normalized length computation}
\label{subsec:normalized_length}
Let us work in the upper half space model $\mb{H}^3=\mb{C}\times(0,\infty)$. Denote by $\rho:\pi_1(\partial\mb{T}(\gamma))\to{\rm Isom}^+(\mb{H}^3)$ the holonomy homomorphism. We normalize the configuration so that the image fixes $\infty$. Under this assumption, each $\rho(\eta)$ is a parabolic transformation of the form $(z,t)\to(z+\tau(\eta),t)$, where $\tau:\pi_1(\partial\mb{T}(\gamma))\to\mb{C}$ is a group homomorphism, and we can identify $\mb{T}(\gamma)$ with the quotient by $\rho(\pi_1(\partial\mb{T}(\gamma)))$ of a horoball
\[
\mc{O}=\{(z,t)\in\mb{H}^3\left|\;t\ge T\right.\}
\]
whose boundary is the horosphere $\mc{H}=\mb{C}\times\{T\}$. Note that the metric on $\mc{H}$ coincides with the standard flat metric of $\mb{C}$ rescaled by $1/T$. Thus, the length of the flat geodesic representative of any $\eta\in\pi_1(\partial\mb{T}(\gamma))$ coincides with $|\tau(\eta)|/T$.

Observe that the geodesic representatives of $\alpha$ and $\beta$ are orthogonal. This is a consequence of the reflective symmetry of $\mb{M}_\gamma$, as follows. The doubling involution $\sigma$ from Subsection~\ref{subsec:hyperbolization} is an isometry of $\mb{M}_\gamma$; since the maximal cusp neighborhood $\mb{T}(\gamma)$ is canonical, $\sigma$ preserves $\mb{T}(\gamma)$ and restricts to an isometric involution of the flat torus $\partial\mb{T}(\gamma)$. By construction, the fixed point set of $\sigma$ on $\partial\mb{T}(\gamma)$ is the union of the two disjoint simple closed curves $\alpha_1\cup\alpha_2=\partial A$, both in the class of $\alpha$, and $\sigma$ exchanges the two components of their complement, namely the two copies of the annulus $\partial U-A$; see Figure~\ref{fig:cusp_torus}. Being components of the fixed point set of an isometry, $\alpha_1$ and $\alpha_2$ are closed flat geodesics. In other words, $\partial\mb{T}(\gamma)$ is the double, along its geodesic boundary, of the compact flat annulus $\partial U-A$. A compact flat annulus with geodesic boundary is isometric to $[0,c]\times\mb{R}/a\mb{Z}$ for some $c>0$, where $a={\rm Length}(\bar\alpha)$, so its double is the rectangular torus $\mb{R}/2c\mb{Z}\times\mb{R}/a\mb{Z}$. Moreover, the double $\beta=\kappa\cup\sigma(\kappa)$ of an essential arc $\kappa$ of the annulus is homologous to the circle $\mb{R}/2c\mb{Z}\times\{{\rm pt}\}$, since the $\alpha$-windings of $\kappa$ and of its mirror image $\sigma(\kappa)$ cancel. Hence the geodesic representative $\bar\beta$ of $\beta$ is orthogonal to $\bar\alpha$, as claimed. In the notation above, after normalizing so that $\tau(\alpha)\in\mb{R}$, we conclude that $\tau(\beta)\in i\mb{R}$.

We set $a\coloneqq{\rm Length}(\bar{\alpha})=\frac{|\tau(\alpha)|}{T}$ and $b\coloneqq{\rm Length}(\bar{\beta})=\frac{|\tau(\beta)|}{T}$. By orthogonality, the length of the flat geodesic representative of $\mu_n=n\alpha+\beta$ and the area of the torus can be computed exactly:
\begin{align*}
{\rm Length}(\bar\mu_n)^2 &= \frac{|\tau(n\alpha+\beta)|^2}{T^2} = \frac{n^2|\tau(\alpha)|^2+|\tau(\beta)|^2}{T^2} = n^2a^2+b^2\quad\text{and} \\
{\rm Area}(\partial\mb{T}(\gamma)) &= \frac{|\tau(\alpha)||\tau(\beta)|}{T^2} = ab.
\end{align*}
Therefore, the normalized length of $\mu_n$ is given by
\begin{equation}
\label{eq:bound}
L(\mu_n)^2=\frac{n^2a^2+b^2}{ab}=n^2\frac{a}{b}+\frac{b}{a}.
\end{equation}

As a consequence, we want to bound the lengths $a$ and $b$.
Since $\alpha$ and $\beta$ are essential in $\partial\mb{T}(\gamma)$, the waist size bound from Subsection~\ref{subsec:the_cusp} yields the lower bounds $a,b\geq 2r$. We will use the following upper bounds.

\begin{pro}\label{prop}
Let $H$ be a 
handlebody of genus $g\geq 2$. For every essential simple closed curve $\gamma\subset\partial H$ such that $(H,\gamma)$ is {pared acylindrical}, we have
\[{\rm Length}(\bar{\alpha})\leq 6 (g-1) \text{ and } {\rm Length}(\bar{\beta})\leq 4,\]
where $\bar{\alpha}$ and $\bar{\beta}$ are geodesic representatives  of the curves $\alpha,\beta\subset
\partial\mb{T}(\gamma)$ as defined in Subsection \ref{subsec:the_cusp}.
\end{pro}
Using Proposition~\ref{prop}, which we prove in the next section, we prove Theorem~\ref{thm}. 

\begin{proof}[Proof of Theorem~\ref{thm}]
Set $t\coloneqq\frac{a}{b}$, so that~\eqref{eq:bound} reads $L^2=n^2t+t^{-1}$, where $L\coloneqq L(\mu_n)$. By Proposition~\ref{prop} and the waist size lower bound $a,b\geq 2r$ from Subsection~\ref{subsec:the_cusp}, we have
\[
\frac{r}{2}=\frac{2r}{4}\leq t\leq\frac{6(g-1)}{2r}=\frac{3(g-1)}{r}.
\]
Note that the function $t\mapsto n^2t+t^{-1}$ is increasing for $t\geq\frac{1}{n}$, and that $\frac{1}{n}\leq\frac{1}{14}<\frac{r}{2}$.

We first bound $L^2$ from below. By monotonicity and $t\geq\frac{r}{2}$, we have
\[
L^2\geq n^2\frac{r}{2}+\frac{2}{r}\geq 14^2\cdot\frac{r}{2}+\frac{2}{r}>61.6>7.823^2,
\]
hence, Theorem~\ref{thm:filling} applies. Combined with $\ell_{M_{\tau_\gamma^n}}(\gamma)<\frac{2\pi}{L^2-28.78}$ from Theorem~\ref{thm:filling}, we have
\[
\ell_{M_{\tau_\gamma^n}}(\gamma)<\frac{2\pi}{n^2\frac{r}{2}+\frac{2}{r}-28.78}.
\]

Secondly, we bound $L^2$ from above to obtain the lower bound for $\ell_{M_{\tau_\gamma^n}}(\gamma)$. By monotonicity and $t\leq\frac{3(g-1)}{r}$, we have
\[
L^2\leq n^2\frac{3(g-1)}{r}+\frac{r}{3(g-1)}\leq n^2\left(\frac{3(g-1)}{r}+\frac{r}{3}\right).
\]
Together with $\frac{2\pi}{L^2+16.17}<\ell_{M_{\tau_\gamma^n}}(\gamma)$ from Theorem~\ref{thm:filling}, we conclude
\[
\ell_{M_{\tau_\gamma^n}}(\gamma)>\frac{2\pi}{n^2\left(\frac{3(g-1)}{r}+\frac{r}{3}\right)+16.17},
\]
as desired.
\end{proof}

\section{The proof of Proposition~\ref{prop}}\label{sec:proofprop}

We use the following elementary fact about 3-dimensional hyperbolic space.

\begin{lem}\label{lem:2.5}
 Consider a path $\omega$ in $\mathbb H^3$ consisting of a concatenation $\lambda_0\star\nu_0\star\dots\star\lambda_N\star\nu_N$ and horospheres $\mc{H}_0,\dots, \mc{H}_{N+1}$ where
 \begin{itemize}
  \item $\lambda_i$ is a geodesic on the horosphere $\mc{H}_i$ of length $>2$, where the horoballs bounded by $\mc{H}_i$ and $\mc{H}_{i+1}$ have disjoint interiors, and
  \item $\nu_i$ is a geodesic segment that meets $\mc{H}_i$ and $\mc{H}_{i+1}$ orthogonally (a single point if the two horospheres are tangent).
 \end{itemize}
Then $\omega$ is not a loop.
\end{lem}
In a first version of the article, we established this lemma with the constant $2.5$ in place of $2$. We are thankful to an anonymous referee for pointing out that the argument can be improved to yield the constant $2$, using a lemma of Adams~\cite{Adams95}.

In fact, the constant $2$ is optimal, as the following example shows. In the original version of this paper, the example below was off by a factor of 2 (compare the last line in the example), an error that was spotted by Claude AI~\cite{claude2026}.

\begin{ex}\label{ex:optimal}
We work in the upper half space model $\mb{H}^3\subset \mb{C}\times \mb{R}$, where the boundary is understood to be $\mb{C}\times\{0\}\cup\infty$. We first record the following elementary fact. Let $\mc{H}$ be the horosphere of Euclidean diameter $\eta$ with ideal point $u\in\mb{C}$, and let $v,w\in\mb{C}$ be two further ideal points. Then the feet on $\mc{H}$ of the geodesics from $u$ to $v$ and from $u$ to $w$ are at distance
\[\eta\,\frac{|v-w|}{|v-u|\,|w-u|}\]
in the intrinsic (flat) metric of $\mc{H}$. Indeed, after translating $u$ to $0$, the inversion in the unit sphere centered at $0$ is an isometry of $\mb{H}^3$ mapping $\mc{H}$ to the horizontal plane $\mb{C}\times\{1/\eta\}$ and the two geodesics to the vertical lines at $1/\bar{v}$ and $1/\bar{w}$; the claim follows since $\left|\frac{1}{\bar v}-\frac{1}{\bar w}\right|=\frac{|v-w|}{|v||w|}$ and the intrinsic metric of $\mb{C}\times\{1/\eta\}$ is the Euclidean one rescaled by $\eta$.

Now fix $m\geq 3$, let $v_k\coloneqq e^{2\pi ik/m}$ for $k=0,\dots,m-1$ be the vertices of a regular ideal $m$-gon, and center at each $v_k$ a horosphere $\mc{H}_k$ of Euclidean diameter $\eta\coloneqq 2\sin(\pi/m)$, so that consecutive horospheres are tangent. On each $\mc{H}_k$, connect the two tangency points with the neighboring horospheres by a horospherical geodesic $\lambda_k$. This yields a closed concatenation as in Lemma~\ref{lem:2.5}, degenerate in that the $\nu_k$ are points. Since $|v_{k\pm1}-v_k|=2\sin(\pi/m)$ and $|v_{k+1}-v_{k-1}|=2\sin(2\pi/m)$, the fact above gives
\[{\rm Length}(\lambda_k)=2\sin(\pi/m)\cdot\frac{2\sin(2\pi/m)}{\left(2\sin(\pi/m)\right)^2}=\frac{\sin(2\pi/m)}{\sin(\pi/m)}=2\cos(\pi/m).\]
Since $2\cos(\pi/m)\to 2$ as $m\to\infty$, for every $c<2$ there is a concatenation forming a loop in which all horospherical pieces have length greater than $c$. Hence the constant $2$ in Lemma~\ref{lem:2.5} cannot be replaced by any smaller constant. (For $m=3$, the three mutually tangent horoballs centered at the vertices of an ideal triangle give sides of length exactly $1$.)
\end{ex}
\begin{proof}[Proof of Lemma~\ref{lem:2.5}]
First note that $N\geq 1$ if $\omega$ is a loop: for $N=0$, the segment
$\nu_0$ would meet $\mc{H}_0=\mc{H}_1$ orthogonally at both endpoints,
forcing both of its ideal endpoints to be the center of $\mc{H}_0$,
which is absurd.
If $\omega$ were a loop then it would be a bracelet in the sense of \cite{Adams95}, and we would have $\mc{H}_{N+1}=\mc{H}_0$. Fix an identification of $\mathbb H^3$ with the upper-half space model such that the horoball bounded by $\mc{H}_0$ is identified with $\{(z,t):t\geq 1\}$. Consider the horosphere $\mc{H}_i$, $1\leq i\leq N$, with the smallest Euclidean diameter. If $i\neq 1$ then, since $\mc{H}_{i-1}$ has diameter at least that of $\mc{H}_i$, \cite[Lemma 4.3]{Adams95} says that the starting point of $\lambda_i$ lies within distance $1$ of the north pole of $\mc{H}_i$ (the boundary case of tangent horoballs follows by continuity). This is still true if $i=1$ since in that case the starting point of $\lambda_1$ even coincides with the north pole. Similarly, the final point of $\lambda_i$ lies within distance $1$ of the north pole of $\mc{H}_i$. Therefore, $\lambda_i$ has length at most $2$, contradicting the assumption that every $\lambda_i$ has length $>2$.
\end{proof}

\begin{proof}[Proof of Proposition~\ref{prop}]
We first prove the upper bound for $a\coloneqq{\rm Length}(\bar{\alpha})$. By the Gauss--Bonnet theorem, the hyperbolic surface $\partial H -\gamma$ has area $-2\pi\chi(\partial H -\gamma)=4\pi(g-1)$.

We have that $A\subset\partial H -\gamma$ is the union of two disjoint embedded cusp neighborhoods, each with boundary of length $a$. Since the area of a cusp neighborhood of a hyperbolic surface is equal to the length of its boundary (as can e.g.~be checked by an explicit calculation in the upper-half plane model), we have ${\rm Area}(A)=2a$. By the packing estimate of B\"or\"oczky~\cite{Boroczky}, the density of a union of disjoint embedded cusp neighborhoods in a hyperbolic surface is at most $\frac{3}{\pi}$. Therefore,
\[2a={\rm Area}(A)\leq\frac{3}{\pi}\cdot{\rm Area}(\partial H-\gamma)=\frac{3}{\pi}\cdot 4\pi(g-1)=12(g-1),\]
and so $a\leq 6 (g-1)$.

Next we apply Lemma~\ref{lem:2.5} to get an upper bound for ${\rm Length}(\bar{\beta})$.
Consider an (arbitrary) essential simple closed curve $\delta\subset\partial H$ bounding a properly embedded disk $\Delta\subset H$. By applying an isotopy if necessary, we can and do assume that $\delta$ is in minimal position with respect to $\gamma$ (note that $\delta$ must intersect $\gamma$ because of the pared acylindrical assumption). We write $\delta$ as a concatenation
\[
\delta=\delta_1\star\dots\star\delta_v,\]
where each $\delta_j$ is the closure of a component of $\delta-\gamma$. We label the intersections $\delta\cap\gamma=\{p_1,\dots,p_v\}$ such that $\delta_j$ is the arc joining $p_j$ to $p_{j+1}$ (indices modulo $v$).

Each arc $\delta_j$ is properly homotopic within $\partial H-\gamma$ to a bi-infinite complete geodesic $\bar{\delta}_j$. By the homotopy extension property, we can extend the homotopy on $\partial\Delta-\{p_1,\dots,p_v\}$ to a proper homotopy of the whole (punctured) disk $\Delta-\{p_1,\dots,p_v\}$. Properness of the map $\Delta-\{p_1,\dots,p_v\}\to\mb{M}_\gamma$ implies that every vertex $p_j$ has a small neighborhood $W_j$ that maps to $\mb{T}(\gamma)$. 

We lift the resulting ideal disk with totally geodesic boundary to $\mb{H}^3$ and, by a slight abuse of notation, we will keep denoting its side by $\bar{\delta}_j$. (We caution the reader that this lifted ideal disk is most likely not planar: each ideal edge of its boundary is a geodesic, but the union of those edges need not lie in a single isometric copy of $\mb{H}^2$. Pleated surfaces of this kind, such as the double of this disk, play a big role in the proof of the $6$-theorem; see Remark~\ref{rmk:6thm}.) Note that each $\bar{\delta}_j$ joins the centers $\xi_j,\xi_j'$ of different horoball components $\mc{O}_j,\mc{O}_j'$ of the pre-image of the cusp $\mb{T}(\gamma)$ in the universal cover $\mb{H}^3$. Note also that we must have $\xi_j'=\xi_{j+1}$ (indices modulo $v$) since every small neighborhood $W_j$ is mapped to the same horoball.

For each side $\bar{\delta}_j$ consider the horospheres $\mc{H}_j=\partial\mc{O}_j$ and $\mc{H}_j'=\partial\mc{O}_j'$ centered at $\xi_j$ and $\xi_j'$. Let $x_j,x'_j$ be the intersections of $\mc{H}_j,\mc{H}_j'$ with $\bar{\delta}_j$. 

As $x_j',x_{j+1}$ lie on the same horosphere $\mc{H}_j'=\mc{H}_{j+1}$ (indices modulo $v$), they are connected by a horospherical geodesic $h_j$. Observe that the projection of $h_j$ to $\mb{M}_\gamma$ lies on $\partial U-A$, so the double of the arc represents the homology class of $\beta$. 

In order to conclude, it is enough to bound the length of at least one of the $h_j$. Indeed, since $\beta$ is freely homotopic to the double of any arc $h_j$, we have that ${\rm Length}(\bar{\beta})$ is less than or equal to twice the length of $h_j$ for all $1\leq j\leq v$.

The path
$
[x_1,x_1']\star h_1\star[x_2,x_2']\star h_2\star\dots\star[x_v,x_v']\star h_v
$
is closed, each segment $[x_j,x_j']$ lies on the geodesic $\bar{\delta}_j$, which meets the horospheres $\mc{H}_j$ and $\mc{H}_j'$ orthogonally since it limits to their centers, and the horoballs $\mc{O}_j$ and $\mc{O}_j'$ have disjoint interiors, being distinct components of the pre-image of the embedded cusp $\mb{T}(\gamma)$. Hence Lemma~\ref{lem:2.5} implies that there is a $j$ such that the flat arc $h_j$ has length at most $2$. Thus, ${\rm Length}(\bar{\beta})$ is at most~$4$.
\end{proof}

\begin{rmk}\label{rmk:6thm}
The disk $\Delta\subset H$ appearing in the proof above has the property that the double of $\Delta$ in the double of $H$ is an essential $2$-sphere; hence the double of $\Delta-\gamma$ in $\mb{M}_\gamma$ is an essential punctured sphere. Studying punctured spheres and punctured tori of this type is how the $6$-theorem of Agol~\cite{Agol_00} and Lackenby~\cite{Lackenby_00} is proved. Put differently: since the double of $H$ is reducible, the $6$-theorem immediately implies ${\rm Length}(\bar{\beta})\leq 6$, and the improvement to $4$ obtained above can be attributed to the reflective symmetry of $\mb{M}_\gamma$. See also Hoffman--Purcell~\cite{HoffmanPurcell_17} for work on planar surfaces and potential improvements to the $6$-theorem for reducible fillings.
\end{rmk}

\bibliographystyle{alpha}
\bibliography{bib_arxiv}

@article {Hempel:Heegaard,
    AUTHOR = {Hempel, J.},
     TITLE = {3-manifolds as viewed from the curve complex},
   JOURNAL = {Topology},
  FJOURNAL = {Topology. An International Journal of Mathematics},
    VOLUME = {{\bf 40}},
      YEAR = {2001},
    NUMBER = {3},
     PAGES = {631--657},
      ISSN = {0040-9383},
     CODEN = {TPLGAF},
   MRCLASS = {57N10},
  MRNUMBER = {1838999},
MRREVIEWER = {Andrei Yu. Vesnin},
       DOI = {10.1016/S0040-9383(00)00033-1},
       URL = {http://dx.doi.org/10.1016/S0040-9383(00)00033-1},
}

@article {twist-distance,
    AUTHOR = {Yoshizawa, M.},
     TITLE = {High distance {H}eegaard splittings via {D}ehn twists},
   JOURNAL = {Algebr. Geom. Topol.},
  FJOURNAL = {Algebraic \& Geometric Topology},
    VOLUME = {14},
      YEAR = {2014},
    NUMBER = {2},
     PAGES = {979--1004},
      ISSN = {1472-2747},
   MRCLASS = {57M50},
  MRNUMBER = {3180825},
MRREVIEWER = {Jungsoo Kim},
       DOI = {10.2140/agt.2014.14.979},
       URL = {https://doi.org/10.2140/agt.2014.14.979},
}

@book {Na05,
    AUTHOR = {Namazi, H.},
     TITLE = {Heegaard splittings and hyperbolic geometry},
      NOTE = {Thesis (Ph.D.)--State University of New York at Stony Brook},
 PUBLISHER = {ProQuest LLC, Ann Arbor, MI},
      YEAR = {2005},
     PAGES = {187},
      ISBN = {978-0542-20270-4},
   MRCLASS = {Thesis},
  MRNUMBER = {2707466},
       URL =
              {http://gateway.proquest.com/openurl?url_ver=Z39.88-2004&rft_val_fmt=info:ofi/fmt:kev:mtx:dissertation&res_dat=xri:pqdiss&rft_dat=xri:pqdiss:3179646},
}

@article {NS09,
    AUTHOR = {Namazi, H. and Souto, J.},
     TITLE = {Heegaard splittings and pseudo-{A}nosov maps},
   JOURNAL = {Geom. Funct. Anal.},
  FJOURNAL = {Geometric and Functional Analysis},
    VOLUME = {{\bf 19}},
      YEAR = {2009},
    NUMBER = {4},
     PAGES = {1195--1228},
      ISSN = {1016-443X},
   MRCLASS = {57M50 (57N10)},
  MRNUMBER = {2570321},
MRREVIEWER = {Bruno P. Zimmermann},
       DOI = {10.1007/s00039-009-0025-3},
       URL = {https://doi.org/10.1007/s00039-009-0025-3},
}

@article{FSVa,
 author = {Feller, Peter and Sisto, Alessandro and Viaggi, Gabriele},
 title = {Uniform models and short curves for random 3-manifolds},
 fjournal = {Compositio Mathematica},
 journal = {Compos. Math.},
 issn = {0010-437X},
 volume = {161},
 number = {3},
 pages = {447--502},
 year = {2025},
 language = {English},
 doi = {10.1112/S0010437X24007565},
 zbMATH = {8070176},
 Zbl = {1577.57020},
 NOTE = {ArXiv:1910.09486 [math.GT].}
}

@article{HV,
AUTHOR={Hamenst\"adt, U. and Viaggi, G.},
JOURNAL={ArXiv e-print},
 TITLE={Small eigenvalues of random 3-manifolds},
YEAR={2019},
 NOTE = {ArXiv:1903.08031 [math.GT]},
}

@article {HJ22,
    AUTHOR = {Hamenstädt, U. and J\"ackel, F.},
    TITLE= {Stability of Einstein metrics and effective hyperbolization in large Hempel distance},
	JOURNAL={ArXiv e-print},
       NOTE = {ArXiv:2206.10438 [math.GT].},
}

@article {FSVb,
    AUTHOR = {Feller, P. and Sisto, A. and Viaggi, G. },
    TITLE= {Effective hyperbolization and length bounds for Heegaard splittings},
      JOURNAL={ArXiv e-print},
       NOTE = {ArXiv:2408.06998 [math.GT].},
}

@incollection {Le05,
    AUTHOR = {Lecuire, C.},
     TITLE = {An extension of the {M}asur domain},
 BOOKTITLE = {Spaces of {K}leinian groups},
    SERIES = {London Math. Soc. Lecture Note Ser.},
    VOLUME = {329},
     PAGES = {49--73},
 PUBLISHER = {Cambridge Univ. Press, Cambridge},
      YEAR = {2006},
   MRCLASS = {57N10 (30F40 57M50)},
  MRNUMBER = {2258744},
MRREVIEWER = {Ken-ichi Ohshika},
}

@article{MasurMinsky:I,
  title={Geometry of the complex of curves {I}: {H}yperbolicity},
  author={Masur, H. and Minsky, Y.},
  journal={Invent. Math.},
  volume={{\bf 138}},
  number={1},
  pages={103--149},
  year={1999},
  publisher={Springer}
}

@article{MasurMinsky:II,
  title={Geometry of the complex of curves {II}: {H}ierarchical structure},
  author={Masur, H. and Minsky, Y.},
  journal={Geom. Funct. Anal.},
  volume={{\bf 10}},
  number={4},
  pages={902--974},
  year={2000},
  publisher={Springer}
}

@article {M10,
    AUTHOR = {Minsky, Y.},
     TITLE = {The classification of {K}leinian surface groups {I}: {M}odels
              and bounds},
   JOURNAL = {Ann. of Math.},
  FJOURNAL = {Annals of Mathematics. Second Series},
    VOLUME = {{\bf 171}},
      YEAR = {2010},
    NUMBER = {1},
     PAGES = {1--107},
      ISSN = {0003-486X},
   MRCLASS = {30F40 (20H10 57M50)},
  MRNUMBER = {2630036},
MRREVIEWER = {Athanase Papadopoulos},
       DOI = {10.4007/annals.2010.171.1},
       URL = {https://doi.org/10.4007/annals.2010.171.1},
}

@article {BCM,
    AUTHOR = {Brock, J. and Canary, R. and Minsky, Y.},
     TITLE = {The classification of {K}leinian surface groups, {II}: {T}he
              ending lamination conjecture},
   JOURNAL = {Ann. of Math. (2)},
  FJOURNAL = {Annals of Mathematics. Second Series},
    VOLUME = {{\bf 176}},
      YEAR = {2012},
    NUMBER = {1},
     PAGES = {1--149},
      ISSN = {0003-486X},
   MRCLASS = {57M50 (30F40)},
  MRNUMBER = {2925381},
MRREVIEWER = {Athanase Papadopoulos},
       DOI = {10.4007/annals.2012.176.1.1},
       URL = {https://doi.org/10.4007/annals.2012.176.1.1},
}

@article {HK:universal,
    AUTHOR = {Hodgson, C. and Kerckhoff, S.},
     TITLE = {Universal bounds for hyperbolic {D}ehn surgery},
   JOURNAL = {Ann. of Math.},
  FJOURNAL = {Annals of Mathematics. Second Series},
    VOLUME = {{\bf 162}},
      YEAR = {2005},
    NUMBER = {1},
     PAGES = {367--421},
      ISSN = {0003-486X},
   MRCLASS = {57M50 (57N10)},
  MRNUMBER = {2178964},
MRREVIEWER = {Joan Porti},
       DOI = {10.4007/annals.2005.162.367},
       URL = {https://doi.org/10.4007/annals.2005.162.367},
}

@article {FPS22,
    AUTHOR = {Futer, D. and Purcell, J. and Schleimer, S.},
     TITLE = {Effective bilipschitz bounds on drilling and filling},
   JOURNAL = {Geom. Topol.},
  FJOURNAL = {Geometry \& Topology},
    VOLUME = {{\bf 26}},
      YEAR = {2022},
    NUMBER = {3},
     PAGES = {1077--1188},
      ISSN = {1465-3060},
   MRCLASS = {57K10 (30F40 57K32)},
  MRNUMBER = {4466646},
       DOI = {10.2140/gt.2022.26.1077},
       URL = {https://doi.org/10.2140/gt.2022.26.1077},
}

@article {Adams_19_WaistSize,
    AUTHOR = {Adams, C.},
     TITLE = {Waist size for cusps in hyperbolic 3-manifolds {II}},
   JOURNAL = {Geom. Dedicata},
  FJOURNAL = {Geometriae Dedicata},
    VOLUME = {203},
      YEAR = {2019},
     PAGES = {53--66},
      ISSN = {0046-5755},
   MRCLASS = {57M50},
  MRNUMBER = {4027583},
MRREVIEWER = {Thilo Kuessner},
       DOI = {10.1007/s10711-019-00425-5},
       URL = {https://doi.org/10.1007/s10711-019-00425-5},
}

@article{Boroczky,
  title={Packing of spheres in spaces of constant curvature},
  author={B{\"o}r{\"o}czky, K{\'a}roly},
  journal={Acta Mathematica Academiae Scientiarum Hungarica},
  volume={32},
  pages={243--261},
  year={1978},
  publisher={Springer}
}

@misc{claude2026,
  author       = {{Anthropic}},
  title        = {{C}laude},
  year         = {2026},
  howpublished = {\url{https://claude.ai}},
  note         = {Large Language Model}
}

@article {Thu86,
    AUTHOR = {Thurston, W.},
     TITLE = {Hyperbolic structures on 3-manifolds, {I}: {D}eformation
              of acylindrical manifolds},
   JOURNAL = {Ann. of Math.},
  FJOURNAL = {Annals of Mathematics. Second Series},
    VOLUME = {{\bf 124}},
      YEAR = {1986},
    NUMBER = {2},
     PAGES = {203--246},
      ISSN = {0003-486X},
   MRCLASS = {57N10},
  MRNUMBER = {855294},
MRREVIEWER = {G. Peter Scott},
       DOI = {10.2307/1971277},
       URL = {https://doi.org/10.2307/1971277},
}

@book {Kapovich:hyperbolization,
    AUTHOR = {Kapovich, M.},
     TITLE = {Hyperbolic manifolds and discrete groups},
    SERIES = {Modern Birkh\"{a}user Classics},
      NOTE = {Reprint of the 2001 edition},
 PUBLISHER = {Birkh\"{a}user Boston, Inc., Boston, MA},
      YEAR = {2009},
     PAGES = {xxviii+467},
      ISBN = {978-0-8176-4912-8},
   MRCLASS = {57M50 (20F65 20H10 30F40 30F45 30F60 32G15)},
  MRNUMBER = {2553578},
       DOI = {10.1007/978-0-8176-4913-5},
       URL = {https://doi.org/10.1007/978-0-8176-4913-5},
}

@article {BonahonOtal_04,
    AUTHOR = {Bonahon, F. and Otal, J.-P.},
     TITLE = {Laminations measur\'{e}es de plissage des vari\'{e}t\'{e}s hyperboliques
              de dimension 3},
   JOURNAL = {Ann. of Math. (2)},
  FJOURNAL = {Annals of Mathematics. Second Series},
    VOLUME = {{\bf 160}},
      YEAR = {2004},
    NUMBER = {3},
     PAGES = {1013--1055},
      ISSN = {0003-486X},
   MRCLASS = {57M50 (57N10 57R30)},
  MRNUMBER = {2144972},
MRREVIEWER = {Thilo Kuessner},
       DOI = {10.4007/annals.2004.160.1013},
       URL = {https://doi.org/10.4007/annals.2004.160.1013},
}

@article {lickorish62,
    AUTHOR = {Lickorish, W. B. R.},
     TITLE = {A representation of orientable combinatorial {$3$}-manifolds},
   JOURNAL = {Ann. of Math. (2)},
  FJOURNAL = {Annals of Mathematics. Second Series},
    VOLUME = {76},
      YEAR = {1962},
     PAGES = {531--540},
      ISSN = {0003-486X},
   MRCLASS = {54.78},
  MRNUMBER = {151948},
MRREVIEWER = {A. D. Wallace},
       DOI = {10.2307/1970373},
       URL = {https://doi.org/10.2307/1970373},
}

@article{Adams95,
 author = {Adams, Colin},
 title = {Unknotting tunnels in hyperbolic 3-manifolds},
 fjournal = {Mathematische Annalen},
 journal = {Math. Ann.},
 issn = {0025-5831},
 volume = {302},
 number = {1},
 pages = {177--195},
 year = {1995},
 language = {English},
 doi = {10.1007/BF01444492},
 url = {https://eudml.org/doc/165328},
 zbMATH = {769474},
 Zbl = {0830.57009}
}

@article{Agol_00,
 author = {Agol, Ian},
 title = {Bounds on exceptional {Dehn} filling},
 fjournal = {Geometry \& Topology},
 journal = {Geom. Topol.},
 issn = {1465-3060},
 volume = {4},
 pages = {431--449},
 year = {2000},
 language = {English},
 doi = {10.2140/gt.2000.4.431},
 url = {https://eudml.org/doc/120912},
 zbMATH = {1533629},
 Zbl = {0959.57009}
}

@article{Lackenby_00,
 author = {Lackenby, Marc},
 title = {Word hyperbolic {Dehn} surgery},
 fjournal = {Inventiones Mathematicae},
 journal = {Invent. Math.},
 issn = {0020-9910},
 volume = {140},
 number = {2},
 pages = {243--282},
 year = {2000},
 language = {English},
 doi = {10.1007/s002220000047},
 zbMATH = {1463429},
 Zbl = {0947.57016}
}

@article{HoffmanPurcell_17,
 author = {Hoffman, Neil R. and Purcell, Jessica S.},
 title = {Geometry of planar surfaces and exceptional fillings},
 fjournal = {Bulletin of the London Mathematical Society},
 journal = {Bull. Lond. Math. Soc.},
 issn = {0024-6093},
 volume = {49},
 number = {2},
 pages = {185--201},
 year = {2017},
 language = {English},
 doi = {10.1112/blms.12000},
 zbMATH = {6774543},
 Zbl = {1378.57026}
}

\end{document}